\documentclass[11pt]{article}

\usepackage[latin1]{inputenc}

\usepackage{amsmath,amssymb,amsthm,graphicx}
\usepackage{amsfonts}
\usepackage{subfig}
\usepackage{color}
\usepackage{authblk} %A package for multiple affiliations

\newtheorem{theorem}{Theorem}[section]
\newtheorem{proposition}[theorem]{Proposition}

\newtheorem{lemma}[theorem]{Lemma}
\newtheorem{conjecture}[theorem]{Conjecture}
\newtheorem*{theoremx}{Theorem}
\newtheorem*{propositionx}{Proposition}

\theoremstyle{remark}
\newtheorem{remark}[theorem]{Remark}
\theoremstyle{definition}
\newtheorem{definition}[theorem]{Definition}
\newtheorem{example}[theorem]{Example}

\numberwithin{equation}{section}
\numberwithin{figure}{section}

\newcommand{\R}{\mathbb{R}}
\newcommand{\Z}{\mathbb{Z}}
\newcommand{\N}{\mathbb{N}}

\newcommand{\A}{\mathcal{A}}

\newcommand{\FF}{\mathcal{F}}
\newcommand{\LL}{\mathcal{L}}

\newcommand{\F}{\mathbb{F}}

\renewcommand{\emptyset}{{\varnothing}}
\DeclareMathOperator{\conv}{conv}

\DeclareMathOperator{\sd}{sd}
\DeclareMathOperator{\rk}{rk}
\renewcommand{\bar}{\overline}

\newcommand{\bil}[1]{\ensuremath{(#1)}}

\newcommand{\heading}[1]{\medskip\par\noindent{\bf #1}}

\title{Non-embeddability of geometric lattices and buildings}

\author[1,2,a]{Martin Tancer}
\author[2,b]{Kathrin Vorwerk}

\affil[1]{\small Department
of Applied Mathematics,
Charles University in Prague, Malostransk\'{e} n\'{a}m.
25, 118~00~~Praha~1.}
\affil[2]{Institutionen f\"{o}r matematik, Kungliga Tekniska H\"{o}gskolan, 100~44 Stockholm}
\affil[a]{\small Partially supported by the ERC Advanced Grant No.
267165 and by the Center of Excellence--Inst.\ for Theor.\ Comp.\ Sci., Prague
(project P202/12/G061 of~GA~\v{C}R).}

\affil[b]{\small Supported by grant KAW 2005.0098 from the Knut and Alice Wallenberg Foundation.}

\date{}
\begin{document}

\maketitle

%\centerline{Manuscript in preparation}

\begin{abstract}
A fundamental question for simplicial complexes is to find the lowest dimensional Euclidean space in which they can be embedded.
We investigate this question for order complexes of posets. We show that order complexes of thick geometric lattices as well as several classes of finite buildings, all of which are order complexes, are hard to embed.
That means that such $d$-dimensional complexes require
$(2d + 1)$-dimensional Euclidean space for an embedding.
(This dimension is in general always sufficient for any $d$-complex.)

We develop a method to show non-embeddability for general order complexes of posets which builds on properties of the van Kampen obstruction.
\end{abstract}

%****************************************************************************
%****************************************************************************
\section{Introduction}
%****************************************************************************

%\heading{Embeddability.}
A classical question for \emph{simplicial complexes}
%\footnote{We will recall necessary definitions in Section~\ref{s:prel}.}
is to find the smallest integer
$m$ such that the complex $K$ \emph{embeds} into $\R^m$. It has been studied
%in topology
since the 1930's. 

It is not hard to see that any $d$-dimensional simplicial complex can be
embedded even linearly into $\R^{2d+1}$ by putting its vertices on the moment
curve. On the other hand, there are $d$-dimensional simplicial complexes which
do not embed into $\R^{2d}$. Basic examples are known as the van
Kampen-Flores~complexes~\cite{vankampen32,flores32}.
They are the $d$-dimensional skeleton of a $(2d+2)$-dimensional simplex and the $(d+1)$-fold join of a three-point discrete set.
%The latter one, denoted by $D^{*(d+1)}_3$,
%will play an important role in this paper.

We investigate the question of embeddability for
%a special class of simplicial complexes coming as
order complexes of posets.
%(partially ordered sets).
%Namely,
We develop a general method with which one can show that
certain order complexes of
%\emph{thick}
posets do not embed into Euclidean space of low dimension.
%We focus particularly on \emph{geometric lattices} and
%(Tits') \emph{buildings}.
%
We first apply this method to order complexes of finite subspace lattices.
%Our first result considers order complexes of subspace lattices.
%\begin{propositionx}[Proposition \ref{p:typeA}]
%If $\Delta$ is the order complex of the lattice of subspaces of the finite projective space $\F^{d+2}_q$ with $d \geq 2$, then $\Delta$ does not embed into $\R^{2d}$.
%\end{propositionx}
%Our main results are the following two non-embeddability theorems.
Then, we generalize in two directions: to order complexes of thick geometric lattices and to some classes of finite buildings, all of them being order complexes of posets.
Here, we overview our main results, although some definitions are given in
later sections.

\begin{theoremx}[Theorem \ref{t:geometric}]
If $L$ is a finite thick geometric lattice of rank $d+2$
then the ($d$-dimensional) order complex $\Delta(L)$ does not embed into $\R^{2d}$.
\end{theoremx}

Finite buildings are very symmetric discrete structures with high complexity. A reader familiar with buildings should expect them to be hard to embed into low-dimensional Euclidean space.

\begin{theoremx}[Theorem \ref{t:buildings}]
A $d$-dimensional thick finite
building $\Delta$ does not embed into $\R^{2d}$ if any of the following
conditions is satisfied
\begin{enumerate}
\item $d = 1$,
\item $\Delta$ is of type $A$,
\item $\Delta$ is of type $B$ coming from an
alternating bilinear form on $\F_q^{2d+2}$, or
\item $\Delta$ is of type $B$ coming from an
Hermitian form on $\F^{2d + 2}_{q^2}$ or $\F^{2d + 3}_{q^2}$.
\end{enumerate}
\end{theoremx}

%We note that finite buildings of type A correspond to order complexes of finite projective spaces.

% \martin{TODO: Text in later sections should be adjusted in such a way that
% refers that we prove Theorem~\ref{t:buildings}. I am sorry that it is not done
% yet. I experience minor problems about deciding how to organize the text into
% the sections. Now, it is little bit confusing to understand what is done where.
% I will think about an alternative organization, for example, preliminaries are
% somewhat extensive, or another question is where to write something about
% buildings in order not to distribute it among many sections. It is very well
% possible that I will find out that I am not aware of better organization than
% the current one. I am sorry if even this text is puzzling.}

%****************************************************************************
%\heading{Our method.}
%****************************************************************************
%Before we proceed with a detailed proof of
%Theorems~\ref{t:geometric}~and~\ref{t:buildings}, we briefly describe our
%method. 

Our proof method builds on properties of the van Kampen obstruction.
% Let $\Delta$ be the $d$-dimensional order complex from
% Theorem~\ref{t:geometric} or the $d$-dimensional building from
% Theorem~\ref{t:buildings}.
This is an effectively algorithmically computable cohomological obstruction $\vartheta(\Delta)$ which can be used as certificate for non-embeddability:
%called \emph{Van Kampen obstruction}
If $\vartheta(\Delta) \neq 0$ for some $d$-dimensional simplicial complex $\Delta$, then $\Delta$ does not
embed into $\R^{2d}$
%In addition, the converse is true if $d \neq 2$.
%that is, if $\vartheta(\Delta) = 0$, then $\Delta$ embeds into $\R^{2d}$.
(for $d \neq 2$, the converse is also true). Therefore it would be sufficient to check that
$\vartheta(\Delta) \neq 0$ in order to prove non-embeddability
into $\R^{2d}$.

Unfortunately, it is not a priori obvious how to compute this obstruction for the infinite class of complexes which we consider.
So instead of computing $\vartheta(\Delta)$, we prove and use the following property
where $|K|$ stands for the geometric realization of $K$.

%fact that one can show non-embeddability of $\Delta$ by finding another simplicial complex $K$, for which $\vartheta(K) \neq 0$ is known, and a map $f : |\Delta| \rightarrow |K|$ between the geometric realizations of $\Delta$ and $K$ which satisfies certain properties:

\begin{propositionx}[Proposition \ref{p:map}]
%\label{p:map}
Let $K$ be a $d$-dimensional simplicial complex with
$\vartheta(K) \neq 0$.
Let $L$ be a simplicial complex and $f\colon |K| \rightarrow |L|$ 
be a map satisfying the following condition:
\begin{equation*}
\text{For every two disjoint } \sigma, \tau \in K \text{ we have } f(|\sigma|) \cap f(|\tau|) = \emptyset.
\end{equation*}
Then $L$ does not embed into
$\R^{2d}$.
\end{propositionx}

%In many cases, one can even find an injective such map $f$. Then the barycentric subdivision of $K$ is isomorphic to a subcomplex of $\Delta$ which immediately implies non-embeddability of $\Delta$.

When applying Proposition \ref{p:map} to order complexes of thick posets, we choose $K$ as the van Kampen-Flores complex $D^{*(d+1)}_3$ for which already van Kampen computed that $\vartheta(D^{*(d+1)}_3)$ does not vanish.

We introduce the notion of {\em weakly independent atom configurations} which constitute our main tool for proving Theorem \ref{t:geometric} and Theorem \ref{t:buildings}.
We show that if a poset $P$ contains a weakly independent atom configuration of $3(d+1)$ atoms, then its order complex $\Delta(P)$ cannot be embedded into $\R^{2d}$.
%
%Those are certain sets of atoms in a poset $P$ such that a map $f : |D^{*(d+1)}_3| \rightarrow |\Delta(P)|$ is induced in a natural way which satisfies the conditions of Proposition \ref{p:map}.
%
%\begin{propositionx}[Proposition \ref{p:Xnonemb}]
%If the poset $P$ contains a weakly independent atom configuration
%$\{ x_{i,j}: i \in [d+1], j \in [3] \}$,
%then $\Delta(P)$ is not embeddable in $\R^{2d}$.
%\end{propositionx}

%This is our main tool which applied to the corresponding posets enables us to prove Theorem \ref{t:geometric} and Theorem \ref{t:buildings}.

% In our approach we do not try to compute $\vartheta(\Delta)$ immediately. We
% rather pick some $d$-dimensional complex $K$ with $\vartheta(K) \neq 0$ and
% exhibit a map $f\colon |K| \rightarrow |\Delta|$ with certain properties (here
% $|\cdot|$ denotes the \emph{geometric realization}).
% Usually, we think of $K =
% D^{*(d+1)}_3$.
% % In some cases, we can even find such an injective $f$, and then
% % a copy of $K$ is just a subcomplex of $\Delta$.
% In other cases we, however,
% do not find an injective map but using some properties of van Kampen
% obstruction we can still prove nonembeddability of $\Delta$ into $\R^{2d}$.
% Namely, our trick is to use the following result.

%We prove Proposition~\ref{p:map} in Section~\ref{s:embed}.

%\heading{Outline of the paper.}

The article is organized as follows:
In Section \ref{s:prel}, we recall necessary definitions about simplicial complexes, posets and order complexes, and we give a short introduction to geometric lattices.

Embeddability and the van Kampen obstruction are discussed in Section \ref{s:embed}, including a proof of Proposition \ref{p:map}.

Section \ref{s:ordcom} is devoted to our main tools for showing non-embeddability of order complexes. First, we develop a method to apply Proposition \ref{p:map} for general complexes $K$. Then, we restrict ourselves to $K = D^{*(d+1)}_3$ and weakly independent atom configurations in Section \ref{s:vecconf}, and we prove the non-embeddability result for order complex of posets containing a weakly independent atom configuration.

In the remaining sections, we apply our methods to some classes of order complexes and show the main theorems. In Section \ref{s:projspace}, we prove Theorem \ref{t:typeA} about order complexes of subspace lattices of projective spaces. Theorem \ref{t:geometric} for order complexes of thick geometric lattices is proved in Section \ref{s:geolatt}.
Finally, we prove Theorem \ref{t:buildings} for buildings in Section \ref{s:build} including a short introduction to finite buildings in Section \ref{s:buildings}.

%****************************************************************************
%****************************************************************************
\section{Preliminaries}
\label{s:prel}
%****************************************************************************

\heading{Miscellaneous.}
By $[n]$ we denote the set $\{1, \dots, n\}$.
Given vectors $v_1, \dots, v_k$
the symbol $\langle v_1, \dots, v_k\rangle$ denotes their span.

%****************************************************************************
%\subsection{Simplicial complexes, geometric realizations and barycentric subdivisions}
%****************************************************************************

%****************************************************************************
\heading{Simplicial complexes.}
%****************************************************************************
We assume knowledge of the basic definitions for abstract simplicial complexes which can be found in Chapter $1$ of \cite{Matousek2003}.

%An \emph{(abstract) simplicial complex} $K$ is a hereditary set system, that means if
%a \emph{face} $\alpha$ belongs to $K$ and $\beta \subset \alpha$ then $\beta$
%belongs to $K$ as well. We work only with finite simplicial complexes. The set
%of \emph{vertices} of $K$ is defined as $V(K) = \bigcup K$. The dimension of a face $\alpha
%\in K$ is defined as $\dim \alpha := \# \alpha - 1$ where $\#$ denotes the number of elements
 %(we keep the symbol $|\alpha|$ for the geometric realization defined
%later). The dimension of $K$, $\dim K$, is the maximal dimension of a face
%contained in $K$.

For two simplicial complexes $K$ and $L$ with disjoint vertex sets $V(K)$ and $V(L)$,
%with $V(K) \cap V(L) = \emptyset$
their \emph{join} $K * L$ is the simplicial complex which has
faces $\alpha \cup \beta$ for all $\alpha \in K$ and $\beta \in L$.
If $V(K)$ and $V(L)$ are not disjoint, we consider artificial copies of $K$ and $L$ on disjoint vertex sets when forming the join.
Thus, it makes sense to speak of an $n$-fold join of a single
complex $K$ which then is the join of $n$ different copies of $K$.
% such as $D^{*n}_3$ mentioned in the introduction.
%In sequel, we refer
%to join of simplicial complexes as to \emph{simplicial join}. (We have to
%distinguish it from join for lattices.)

Note that we use the term {\em join} also in a different context when dealing with posets. The terminology is well-established in both cases and the context should always make it clear which meaning of join we refer to.

%****************************************************************************
\heading{Geometric realizations.}
%****************************************************************************
Assume that $K$ is a simplicial complex and that $f \colon V(K) \rightarrow \R^d$ is a map such that
$$
    \conv \{ f(a) : a \in \alpha \} \cap \conv \{ f(b) : b \in \beta \}
    = 
    \conv \{ f(a) : a \in \alpha \cap \beta \}
$$
for all $\alpha, \beta \in K$. 
Geometrically, this condition means that when extending $f$ affinely to faces
of $K$, 
then the images of two faces $\alpha, \beta \in K$ intersect in the
image of $\alpha \cap \beta \in K$. In particular, the images of 
two disjoint faces never intersect.

%Geometrically, this condition means that when extending $f$ affinely to faces of $K$, then the convex hulls corresponding to two disjoint faces will never intersect in their interior.

In that case, we set
$
|\alpha| := \conv \{f(a) : a \in \alpha\}
$
and
$$
|K| := \bigcup\limits_{\alpha \in K} |\alpha|.
$$
We call $|\alpha|$ a geometric realization of the face $\alpha$ and $|K|$ a geometric realization of $K$.
It is a well-known fact that any two geometric realizations of a complex $K$ are homeomorphic so that the notation $|K|$ is non-ambiguous and we can say that $|K|$ is \emph{the geometric realization of $K$}.

% Let $n$ be the number of vertices of $K$. Let $S$ be an affinely independent
% set in $\R^{n-1}$ with $n$ elements. Let us fix an injective map $f$ from $V(K)$ to $S$. The \emph{canonical geometric realization} of a face $\alpha \in K$ is
% defined as
% $$
% |\alpha| := \conv \{f(a) : a \in \alpha\}
% $$
% where $\conv$ is the convex hull.
% The \emph{canonical geometric realization} of $K$ is then
% $$
% |K| := \bigcup\limits_{\alpha \in K} |\alpha|.
% $$.

A map $g\colon V(K) \rightarrow V(L)$ between the sets of
vertices of two simplicial complexes is \emph{simplicial} if $g(\alpha) \in L$
for every $\alpha \in K$.
A simplicial map $g \colon V(K) \rightarrow V(L)$ also has a geometric realization
$|g| \colon |K| \rightarrow |L|$. Let $f_K$ and $f_L$ be maps from the
geometric realizations of $K$ and $L$. Then we set $|g|(f_K(v)) := f_L(g(v))$ for
$v \in V(K)$ and extend this map affinely on every simplex (see~\cite{Matousek2003} for more details).

%****************************************************************************
\heading{Barycentric subdivisions.}
%****************************************************************************
To a simplicial complex, we can associate another complex as follows:
The \emph{face poset} $\FF(K)$ of $K$ is the poset
whose elements are the faces of $K$ except the empty set ordered by inclusion.

The \emph{order complex} of a poset $P$ is the simplicial complex which has the elements of $P$ as vertices and the collection of chains
of $P$
$$
\Delta(P) := \big\{ \{p_1, \ldots, p_\ell\} : p_1 < \ldots < p_\ell, \ p_1, \ldots,
p_\ell \in P \big\}.
$$
as faces.
Given a simplicial complex $K$, we get its \emph{barycentric subdivision} $\sd K$ as
$$
\sd K := \Delta(\FF(K)).
$$
Note that the vertices of $\sd K$ are the non-empty faces of $K$ and the faces of $\sd K$ are chains of faces of $K$.

Judged by the definitions, the geometric realizations of $K$ and
$\sd K$ might be completely unrelated. However, it is very convenient to 
derive a concrete realization $|\sd K|$ from $|K|$ in the way suggested by Figure~\ref{f:ualpha}:
Let $f$ be a map from the definition of the geometric realization of $K$. Recall that $V(\sd(K)) = K \setminus \{\emptyset\}$. Then, for the realization of $\sd(K)$, we map a vertex $\alpha \in V(\sd(K))$ to the barycentre of $\emptyset \neq |\alpha| \subseteq |K|$.
This yields $|K| = |\sd K|$ and also
\begin{equation}
\label{e:unionalpha}
 |\alpha| = \bigcup \Big\{ |\Gamma| : \Gamma \in \sd K, \alpha \in \Gamma; \beta \in
 \Gamma \Rightarrow \beta \subseteq \alpha \Big\}
\end{equation}
for $\alpha \in K \setminus {\emptyset}$.
In particular, $|K|$ and $|\sd K|$ are canonically homeomorphic.

\begin{figure}
\begin{center}
\includegraphics{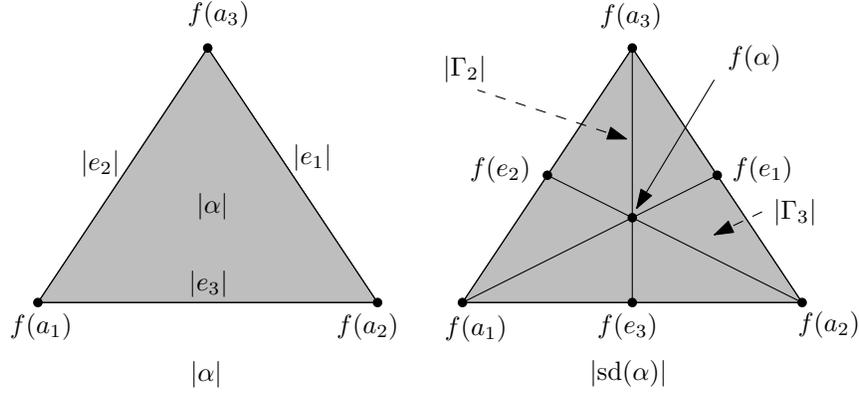}
\caption{
The geometric realization of a simplex $\alpha = \{a_1, a_2, a_3\}$ and
its barycentric subdivision $\sd(\alpha)$.
The geometric realizations of the faces
$\Gamma_2 = \{ a_3, \alpha \}$ and $\Gamma_3 = \{ a_2,e_1,\alpha\}$
of $\sd(\alpha)$
are emphasized.
}
\label{f:ualpha}
\end{center}
\end{figure}

%****************************************************************************
%\subsection{Posets, geometric lattices and modularity}
%****************************************************************************

%****************************************************************************
\heading{Posets, geometric lattices.}
%****************************************************************************
We recall some basic facts about posets and lattices and refer to
\cite[Chapter 3]{White1986} and
\cite[Chapter 4]{Birkhoff1995} for more details.

Let $P$ be a poset. For two elements $a$ and $b$ of $P$, we say that $b$ covers $a$, and write $a \lhd b$, if $a < b$ and if there is no $c$ with $a < c < b$.

The \emph{join} of two elements $a,b \in P$ is the unique smallest element $c \in P$ such that $a \leq c$ and $b \leq c$. Note that joins do not necessarily exist in general posets, but they always exist in lattices by definition.

If $P$ contains a unique minimal (resp. maximal) element, we denote it by
$\hat{0}$ (resp. $\hat{1}$). For a poset $P$, we associate another poset
$\bar{P}$ which is obtained from $P$ by removing the minimal element
$\hat 0$, if it exists, as well as the maximal element $\hat{1}$, if it exists. Recall that the \emph{order complex} of $P$ is the simplicial complex which has the elements of $P$ as vertices and the collection of chains
of $P$ as faces. The {\em reduced order complex} of $P$ is the order complex of $\bar{P}$, that is $\Delta(\bar{P})$.

A poset with minimal element $\hat{0}$ is {\em atomistic} if every element is the join of a set of {\em atoms} (elements which cover $\hat{0}$).
We write $\A(P)$ for the set of atoms of $P$.

A poset is graded if every maximal chain in $P$ has the same length. 
A graded poset $P$ has a {\em rank function} $\rk \colon P \rightarrow \N$.
It is
{\em semimodular} if
\[
    \rk(a) + \rk(b) \geq \rk(a \wedge b) + \rk(a \vee b)
\]
for any two different elements $a,b \in P$.
It is {\em modular} if the condition on the rank function is satisfied with
equality, that is, if 
$\rk(a) + \rk(b) = \rk(a \wedge b) + \rk(a \vee b)$
 for all $a,b \in P$.

A graded lattice of finite rank is {\em geometric} if it is both atomistic and semimodular.

Given two posets $P$ and $Q$, a \emph{poset map} is a map $g\colon P \to Q$
such that $g(p_1) \leq g(p_2)$ if $p_1 \leq p_2$. Such a poset map
induces a simplicial map $\tilde{g}\colon \Delta(P) \to \Delta(Q)$ given by
$\tilde{g}(\{p_1, \dots, p_\ell\}) = \{g(p_1), \dots, g(p_\ell)\}$. (Note that
the image of a simplex is a simplex, possibly of lower dimension.)

The {\em product} of two posets $P$ and $Q$ is the poset $P \times Q$ with relation
$(a,b) \leq (a',b')$
if and only if
$a \leq a'$ and $b \leq b'$ for $a,a' \in P$ and $b,b' \in Q$.
We call a poset {\em irreducible} if it cannot be decomposed as the product of two non-trivial smaller posets.

Note that we also use the term {\em irreducible} both for posets
(product-wise) and buildings (join-wise). As for joins, the context should make it clear which meaning we refer to.

%****************************************************************************
%\heading{Geometric lattices.}
%****************************************************************************

%\red{
%\martin{Preliminaries are very extensive now. I would suggest to move the text
%from now on into later sections. We should also explain, what you wrote me in
%the email. It would only extend this part.}
%}
%\kv{
%I agree. I moved the Lemmas into the result section. I also shortened the simplicial complex part.
%}

%****************************************************************************
%\heading{Modularity.}
%****************************************************************************
%A finite graded lattice $L$ is {\em modular} if the condition on the rank function is satisfied with equality, that means if $r(a) + r(b) = r(a \wedge b) + r(a \vee b)$ holds for all $a,b \in L$.

%****************************************************************************
\section{Embeddability}
%****************************************************************************
\label{s:embed}
%****************************************************************************

In this section, we overview the notions of an embedding and the
van Kampen obstruction.
We also prove Proposition~\ref{p:map}.
(A proof might be obvious to an expert in the field.)
If not interested in the proof, the reader might want to skip this part.
Later, we will only use Proposition~\ref{p:map} and Lemma~\ref{l:joins} from this section.
In many details, we follow~\cite{melikhov09}.
%We state some properties of the van Kampen obstruction and a consequence which will play a central role for our considerations. 

A \emph{topological embedding} (or just \emph{embedding}) of a simplicial complex $K$ into $\R^d$ is
an injective map $f \colon |K| \rightarrow \R^d$.
For finite $K$, the set $|K|$ is compact and then $f$ is a
homeomorphism between $|K|$ and $f(|K|)$.
If there is an embedding of $K$ into
$\R^d$, we say that $K$ is embeddable in $\R^d$.
Concrete geometric realizations of $K$ can be thought
of as linear embeddings of $K$ in some $\R^d$.
For a short overview about differences between topological, linear and piecewise linear embeddings we refer the reader to Chapter~2 and Appendix~C of \cite{matousek-tancer-wagner11}.

\heading{Deleted products and equivariant maps.}
Let $X$ be a compact topological space.
The \emph{deleted product} of $X$ is the
Cartesian product of $X$ without the diagonal:
$$
\tilde{X} := \{(x,y) \in X \times X : x \neq y\}.
$$
The deleted product is equipped with a natural free $\Z_2$-action which
exchanges the coordinates $(x, y) \mapsto (y,x)$. From now on, we always
assume that $\tilde{X}$ denotes the space together with this action, that is, $\tilde{X}$ is a $\Z_2$-space.
We also denote by $S^m_-$ the $m$-sphere
equipped with the antipodal action $x \mapsto -x$. 

Assuming that there exists an embedding $f\colon X \rightarrow \R^m$, we define
the \emph{Gauss map} $\tilde{f} \colon \tilde{X} \rightarrow S^{m-1}_-$ by
$$
\tilde{f}(x,y) := \frac{f(x)-f(y)}{\|f(x) - f(y)\|}.
$$
This map is \emph{equivariant} which means that the $\Z_2$-actions on
$\tilde{X}$ and $S^{m-1}_-$ commute with this map.

%\red{
From now on, we assume that $X$ is the geometric realization of a given
simplicial complex $K$, that means $X = |K|$.
%}

The \emph{simplicial deleted product} of $X$ is again a subspace 
of the Cartesian product, now given by the following formula:
$$
\tilde{X}_s := \{(x,y) \in X \times X : \exists \sigma, \tau \in K; \sigma \cap
\tau = \emptyset; x \in |\sigma|, y \in |\tau| \}.
$$

We note that $\tilde{X}$ and $\tilde{X}_s$ are equivariantly homotopy
equivalent, see the remark below Example~3.3 in~\cite{melikhov09} and the
references therein.

\heading{The van Kampen obstruction.}
Now, we are going to define the van Kampen obstruction.
For shortness, we do not define all notions from cohomology theory that are used.
A reader not familiar with cohomology can skip this definition. 
More details can be found in~\cite{melikhov09}.

Let $X = |K|$ be the geometric realization of some simplicial complex $K$ of dimension $d$.
Let $\bar{X}$ be the quotient space $\tilde{X} / \Z_2$ with respect to the
action on $\tilde{X}$. Similarly the projective space $\R P^{2d-1}$ is the
quotient space $S^{2d-1}_-/ \Z_2$ with respect to the antipodal action.
We also need that the infinite projective space $\R P^{\infty}$ is a
classifying space for $\Z_2$. Then we know that there is unique map up to
homotopy $G\colon \bar{X} \rightarrow \R P^{\infty}$, classifying the line
bundle associated with the double cover $\tilde{X} \rightarrow \bar{X}$. The
\emph{van Kampen obstruction} $\vartheta(X)$ is the element $G^*(\xi) \in
H^{2d}(\bar{X}; \Z)$ where $\xi$ is the generator of $H^{2d}(\R P^{\infty};
\Z) \simeq \Z_2$. If there exists an equivariant map $g\colon \tilde{X}
\rightarrow S^{2d-1}_-$, then $\vartheta(X) = i^*(\bar{g}^*(\xi))$ where
$\bar{g}\colon \bar{X} \rightarrow \R P^{2d-1}$ is the quotient map associated
to $g\colon \tilde{X} \rightarrow S^{2d-1}_-$ and $i\colon \R
P^{2d-1} \hookrightarrow \R P^{\infty}$ is the inclusion. Therefore $\vartheta(X) =
0$ in this case, since $H^{2d} (\R P^{2d-1}; \Z) = 0$.

%\heading{Properties of the van Kampen obstruction.}
The most well-known result about the van Kampen obstruction is the following.
It is mainly based on the work of Shapiro, Wu, and van
Kampen~\cite{shapiro57, wu65, vankampen32}.

\begin{theorem}
\label{th:vk1}
Let $X = |K|$ be the geometric realization of some simplicial complex $K$ of dimension $d$.
\begin{enumerate}
\item If $X$ embeds into $\R^{2d}$ then the Gauss map provides an equivariant map
  $g\colon \tilde{X} \rightarrow S^{2d-1}_-$.
\item If there is an equivariant map $g\colon \tilde{X} \rightarrow
S^{2d-1}_-$, then the van Kampen obstruction $\vartheta(X)$ is zero.
\item If $d \neq 2$ and the van Kampen obstruction $\vartheta(X)$ is zero, then $X$ embeds into $\R^{2d}$.
\end{enumerate}
\end{theorem}

The statements (i) and (ii) in Theorem \ref{th:vk1} follow directly from our discussion before. For the last statement (iii), see Theorem 3.2 and
the text below the proof in~\cite{melikhov09}.

When $X = |K|$, we write $\vartheta(K)$ instead of $\vartheta(|K|)$  keeping in mind that $\vartheta$ only depends on the topological space $|K|$ and not on the concrete triangulation given by $K$.

We need the following related result.

\begin{proposition}
\label{t:prop}
Let $K$ be a $d$-dimensional simplicial complex.
If there exists a map $f \colon |K| \rightarrow \R^{2d}$ such that for every pair of disjoint simplices $\sigma, \tau \in K$ the
intersection $f(|\sigma|) \cap f(|\tau|)$ is empty,
then $\vartheta(K) = 0$.
\end{proposition}

\begin{proof}
Assume that $f \colon |K| \rightarrow \R^{2d}$ as in the proposition exists and set $X = |K|$. The condition on $f$ implies that there is an equivariant map $g\colon \tilde{X}_s \rightarrow S^{2d-1}_-$ defined as the Gauss-map by
\[
  g(x,y) = \frac{f(x)-f(y)}{\|f(x) - f(y)\|}.
\]
However, $\tilde{X}$ and $\tilde{X}_s$ are equivariantly homotopic and we also get an equivariant map $g'\colon \tilde{X} \rightarrow S^{2d-1}_-$. Now, $\vartheta(K) = 0$ follows from part (ii) of Theorem \ref{th:vk1}.
\end{proof}

From Proposition \ref{t:prop}, we can deduce our main tool for showing non-embeddability.

\begin{proposition}
\label{p:map}
Let $K$ be a $d$-dimensional simplicial complex with $\vartheta(K) \neq 0$.
Let $L$ be a simplicial complex and $f\colon |K| \rightarrow |L|$ 
be a map satisfying the following condition:
\begin{equation}
\tag{C}
\label{eq:cond}
\text{For every two disjoint } \sigma, \tau \in K \text{ we have } f(|\sigma|) \cap f(|\tau|) = \emptyset.
\end{equation}
Then $L$ does not embed into
$\R^{2d}$.
\end{proposition}

\begin{proof}%[\red{Proof of Proposition~\ref{p:map}}]
For contradiction, assume that there is an embedding 
$g\colon |L| \rightarrow \R^{2d}$. 
Then $g \circ f$ satisfies the condition on the map from $|K|$ to $\R^{2d}$ in Proposition \ref{t:prop} and therefore
$\vartheta(K) = 0$ in contradiction to our assumption on $K$.
\end{proof}

It is known that the van Kampen obstruction of $(d+1)$-fold join $D_3^{*(d+1)}$ of the three-point discrete set $D_3$ is nonzero \cite{vankampen32}:
$$
 \vartheta(D_3^{*(d+1)}) \neq 0.
$$
Note that $D_3^{*(d+1)}$ is a $d$-dimensional simplicial complex.
% 
% 
% \begin{proposition}[\cite{vankampen32}]
% \label{t:join_emb}
% We have $\vartheta(D_3^{*(d+1)}) \neq 0$.
% \end{proposition}
% 
For our main results about non-embeddability of buildings, we will apply
Proposition \ref{p:map} for $K = D_3^{*(d+1)}$.

It will prove useful later to be able to break down non-embeddability proof to complexes which are irreducible with respect to joins, so we need to make sure our methods work fine when taking simplicial joins of complexes.
We refer the reader to~\cite{Matousek2003} for more details about joins of complexes and maps.

\begin{lemma}
\label{l:joins}
Let $\Delta_1$ and $\Delta_2$ be simplicial complexes. Assume that there exist
complexes $K_i$ and maps $f_i\colon |K_i| \to |\Delta_i|$ which
satisfy condition \eqref{eq:cond} in Proposition \ref{p:map} for $i=1,2$. Then there is
a map $f\colon |K_1 * K_2| \to |\Delta_1 * \Delta_2|$ which satisfies \eqref{eq:cond} as
well.
\end{lemma}

\begin{proof}
We set $f = f_1 * f_2 \colon |K_1 * K_2| \rightarrow |\Delta_1 * \Delta_2|$.

If $\sigma = \sigma_1 * \sigma_2$ and $\tau = \tau_1 * \tau_2$ are disjoint
faces of $K_1 * K_2$, then $\sigma_i$ and $\tau_i$ are disjoint faces of $K_i$ for $i=1,2$. We find that
\begin{align*}
  f(|\sigma|) \cap f(|\tau|)
    &= \left( f_1(|\sigma_1|) * f_2(|\sigma_2|) \right) 
      \cap \left( f_1(|\tau_1|) * f_2(|\tau_2|) \right) \\ 
    &= \left( f_1(|\sigma_1|) \cap f_1(|\tau_1|) \right) 
      * \left( f_2(|\sigma_2|) \cap f_2(|\tau_2|) \right) \\
    &= \emptyset * \emptyset = \emptyset. \qedhere
\end{align*}
%So $\Delta$, $K$ and $f$ satisfy the conditions of Corollary \ref{c:map}.
\end{proof}

%****************************************************************************
\section{Non-embeddable order complexes}
\label{s:ordcom}
%****************************************************************************
%****************************************************************************

%****************************************************************************
\subsection{Non-embeddable order complexes}
\label{s:ordcomsub}
%****************************************************************************

In this section we will develop a method with which one can show that the order complex $\Delta(P)$ of a poset $P$ is not embeddable in $\R^{2d}$ for some $d$ by exposing a subcomplex that is either isomorphic to a known non-embeddable complex or is a weakly degenerated copy of such a complex.

Let $K$ be a simplicial complex and $P$ a poset. Let $g \colon V(K) \rightarrow P$ be a map from the vertices of $K$ to $P$.
We say that $g$ is \emph{extendable} if the join $\bigvee_{x \in \sigma} g(x)$
exists in $P$ for all nonempty faces $\sigma \in K$
and if $\bigvee_{x \in \sigma} g(x)$ is not equal to $\hat{0}$ or $\hat{1}$ for any $\sigma \in K$, if $\hat{0}$ or $\hat{1}$ exist.
If $g$ is extendable then we can extend it to a poset map $g \colon \FF(K) \rightarrow P$ from the face poset $\FF(K)$ of $K$ to $P$ in the canonical way by setting
$$
g(\sigma) := \bigvee_{x \in \sigma} g(x).
$$
(This is a slight abuse of notation because $g$ is a map from the
vertices of $K$ to $P$ and the induced poset map is from the faces of $K$ to
$P$. However, we can identify a vertex $v \in V(K)$ and the face $\{v\} \in K$.) 

The poset map $g : \FF(K) \rightarrow P$ induces a simplicial map
$\tilde{g} \colon \Delta(\FF(K)) \rightarrow \Delta(\bar{P})$
between the order complex $\Delta(\FF(K))$ and the reduced order complex
$\Delta(\bar{P})$
where $\Delta(\FF(K)) = \sd K$ is the barycentric subdivision of $K$.

\begin{definition}
We say that an extendable map $g \colon V(K) \rightarrow P$ is \emph{injective} (resp. \emph{weakly injective}) if the corresponding poset map $g \colon \FF(K) \rightarrow P$ satisfies that
$$
    g(\alpha) \neq g(\beta)
$$
for all $\alpha,\beta \in K$ with $\alpha \neq \beta$ (resp. for all disjoint $\alpha, \beta \in K$).
\end{definition}

As suggested by the notation, all injective extendable maps are also weakly injective.

Because $|\sd K|$ is homeomorphic to $|K|$, we can consider the geometric realization $|\tilde{g}|$ as a map from $|K|$ to $|\Delta(\bar{P})|$.
Because $\tilde{g}$ is simplicial on $\sd K$, the induced map
$|\tilde{g}| \colon |K| \rightarrow |\Delta(\bar{P})|$
is piecewise linear on the geometric realizations $|\sigma|$ of faces $\sigma \in K$.

The following proposition relates weakly injective extendable maps to the
condition stated in Proposition~\ref{p:map}. 

\begin{proposition}
\label{p:disj}
Let $K$ be a simplicial complex and $P$ a poset. If
$g \colon V(K) \rightarrow P$
is a weakly injective map, then the induced map
$|\tilde{g}| \colon |K| \rightarrow |\Delta(\bar{P})|$
satisfies condition \eqref{eq:cond} in
Proposition~\ref{p:map}, that is for every two disjoint $\sigma,\tau \in K$ we have 
$|\tilde{g}|(|\sigma|) \cap |\tilde{g}|(|\tau|) = \emptyset$.
\end{proposition}

\begin{proof}
For contradiction, let $\sigma, \tau \in K$ be disjoint simplices of $K$ such that
\begin{equation}
\label{equ:cond}
 |\tilde{g}|(|\sigma|) \cap |\tilde{g}|(|\tau|) \neq \emptyset.
\end{equation}
Recall that simplices of $\sd K$ are chains of simplices of $K$. We set $\eta(\Gamma)$ to be the maximal simplex of $K$ contained in the chain $\Gamma \in \sd K$.
Following Equation~\eqref{e:unionalpha} from the preliminaries, we then have
$$
 |\sigma| = \bigcup \big\{ |\Gamma| : \Gamma \in \sd K, \; \eta(\Gamma) = \sigma \big\}.
$$
% Because the map $|\tilde{g}|$ on $|K|$ is induced by its behavior on $|\sd K|$,
% {\tt Reference to preliminaries?}
This equation together with assumption \eqref{equ:cond}
implies that there are two chains $\Gamma_1, \Gamma_2 \in \sd K$ such that
$\eta(\Gamma_1) = \sigma$, $\eta(\Gamma_2) = \tau$ and
$$|\tilde{g}|(|\Gamma_1|) \cap |\tilde{g}|(|\Gamma_2|) \neq \emptyset.$$
Since $\tilde{g}$ is simplicial on $\sd K$ and induced by $g \colon \FF(K) \rightarrow P$ (which defines $\tilde{g}$ on the vertices of $\sd K$), we derive that
$$
  \{ g(\alpha) : \alpha \in \Gamma_1 \} \cap \{ g(\beta) : \beta \in \Gamma_2 \} \neq \emptyset.
$$
We thus find $\alpha \in \Gamma_1$ and $\beta \in \Gamma_2$ such
that $g(\alpha) = g(\beta)$. Since $\eta(\Gamma_1) = \sigma$, we have that $\alpha \subseteq \sigma$. Similarly $\beta \subseteq \tau$ and as we assumed $\sigma$ and $\tau$ to be disjoint, also $\alpha$ and $\beta$ are disjoint in contradiction to our assumption on $g$.
\end{proof}

We can now prove the main result of this subsection.

\begin{theorem}
\label{thm:winjmap}
Let $K$ be a $d$-dimensional simplicial complex with $\vartheta(K) \neq 0$ and let $P$ be a poset. If there exists a weakly injective map $g \colon V(K) \rightarrow P$, then
$\Delta(\bar{P})$
is not embeddable in $\R^{2d}$.
\end{theorem}

\begin{proof}
This follows directly from Proposition \ref{p:disj} and Proposition \ref{p:map}.
\end{proof}

We remark that the existence of a simplicial map $\tilde{g}$, as we construct
it from a weakly injective map $g$, implies that $K$ is a homological minor of $\Delta(\bar{P})$ as introduced in \cite{Wagner2011}. Therefore, Theorem \ref{thm:winjmap} can be seen as an application of the result in \cite{Wagner2011} that the existence of a non-embeddable homological minor in a complex implies non-embeddability of the complex itself.

%\martin{I tried to change this example a little bit. I merged the figures
%again. I considered it a little bit less harmful to merge it rather then keep
%two figures closely below each other. Let me know if you want to rearrange the
%objects in the picture. For comparison, I kept the older version (in green), but it should
%be removed if you agree with the change.}

\begin{example}
\label{ex:fanoplane}
Let $K = K_{3,3}$ be the complete bipartite graph on 6 vertices with three
vertices in each part labelled $1,2,3$ and $4, 5, 6$. The barycentric
subdivision of $K$ is shown in Figure~\ref{f:mapg}(a).

Let $e_1,e_2,e_3$ be a basis of $\F_2^3$ and let $P$ be the poset of subspaces of $\F^3_2$ ordered by inclusion. The elements of $\bar{P}$ correspond to the points and lines of the Fano plane.
The order complex $\Delta(\bar{P})$ is a generalized
3-gon where each vertex has degree three, see Figure~\ref{f:mapg}(b).

%{\color{red}
%\ref{f:mapg}(b).
%For simplicity, we repeat it in Figure~\ref{f:mapg}(b) without showing the labels.

%{\tt KV: Part (c) would really get the space it needs for all labels if we drop figure (b) here, as we already had it before. I suggest not showing part (b) here at all and making part (c) larger for better labelling. Move g(34)=g(35) further to the left, so it doesn't intersect the picture. Also maybe move g(15)=g(16) up a little.}
%}

We define a map $g \colon V(K_{3,3}) \rightarrow P$ by
\begin{align*}
  g(1) &= \langle e_1 \rangle,& \quad
  g(2) &= \langle e_2 \rangle,& \quad
  g(3) &= \langle e_1 + e_2 \rangle,
\\
  g(4) &= \langle e_1 + e_3 \rangle,& \quad
  g(5) &= \langle e_2 + e_3 \rangle,& \quad
  g(6) &= \langle e_1 + e_2 + e_3 \rangle.
\end{align*}
Clearly, $g$ is extendable. The image of the induced
%simplicial (Removed this to avoid bad box only!)
map $\tilde{g}
\colon \sd K_{3,3} \rightarrow \Delta(\bar{P})$ is shown in Figure \ref{f:mapg}(c). We invite the reader to check that $g$ is weakly injective.
This shows that $\Delta(\bar{P})$ cannot be embedded in $\R^2$, that means it is a
non-planar graph. (Non-planarity of $\Delta(\bar{P})$ can easily be shown via
Kuratowski's theorem, of course. This example merely serves to demonstrate our
method.)

\begin{figure}
\centering
\includegraphics{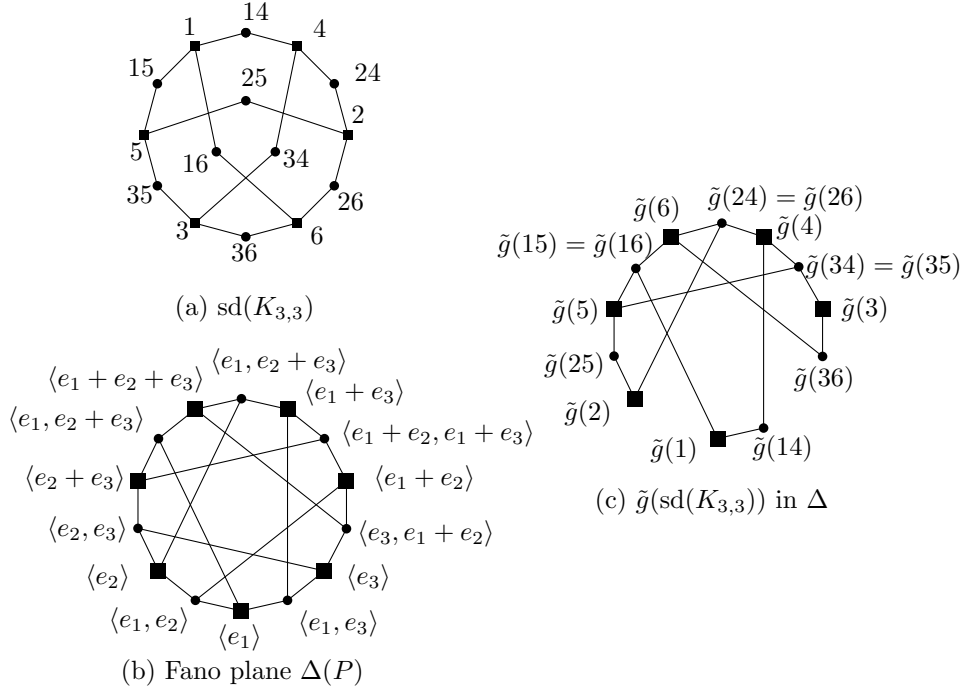}
\caption{Showing non-planarity of the Fano plane}
\label{f:mapg}
\end{figure}
\end{example}

\subsection{Independent and almost independent configurations}
\label{s:vecconf}
%****************************************************************************

We will now consider the case where $K = D_3^{*(d+1)}$. We will reformulate a sufficient condition for the existence of an extendable injective map $g : V(K) \rightarrow P$ in terms of certain collections of atoms in $P$.

Let us fix some $d \geq 1$. Throughout the section, we set $K = D_3^{*(d+1)}$ and denote the vertex set of $D_3^{*(d+1)}$ by
\[
    V \left( D_3^{*(d+1)} \right) = \{ v_{i,j} : i \in [d+1], j \in [3] \},
\]
where the vertex $v_{i,j}$ corresponds to the $j$-th vertex of the $i$-th copy of $D_3$.
The maximal faces of $D_3^{*(d+1)}$ are sets of the form
\[
  \left\{ v_{1,j_1},\ldots,v_{d+1,j_{d+1}} \right\}
    \in \FF \left( D_3^{*(d+1)} \right)
\]
for any choice of $j_1,\ldots,j_{d+1} \in [3]$.

Given a collection
$$
 \left\{ x_{i,j} : i \in [d+1], j \in [3] \right\} \subseteq \A(P)
$$
of {\em atoms},
%(recall that $\A(P)$ denotes the set of atoms of $P$)
we can associate a map $g \colon V(D_3^{*(d+1)}) \rightarrow P$ defined by
$
 g \left( v_{i,j} \right) = x_{i,j}
$.
%We can associate a map $g \colon V(D_3^{*(d+1)}) \rightarrow P$ to every collection
%$$
 %\left\{ x_{i,j} : i \in [d+1], j \in [3] \right\} \subseteq \A(P)
%$$
%of {\em atoms}, where $g$ is defined by
%$
 %g \left( v_{i,j} \right) = x_{i,j}
%$.

As before, the map $g$ is extendable if the join
$$
 \bigvee_{i \in [d+1]} x_{i, j_i}
$$
exists in $P$
and is not equal to the maximal element $\hat{1}$ (if it exists)
for all choices $j_1,\ldots,j_{d+1} \in [3]$.
In that case, we call the collection of atoms $\{ x_{i,j} : i \in [d+1], j \in [3] \}$ an \emph{extendable atom configuration}.

\begin{remark}
In general, it would not be necessary to choose $x_{i,j}$ as atoms in $P$. However, this will always be the case in our applications.
\end{remark}

We will now reformulate the condition of $g$ being (weakly) injective into the setting of extendable atom configurations.

\begin{definition}
\label{def:windconf}
We call an extendable atom configuration
\[
  \{ x_{i,j}: i \in [d+1], j \in [3] \} \subseteq \A(P)
\]
{\em independent} (resp. {\em weakly independent}) if the corresponding map
$g \colon V \left( D_3^{*(d+1)} \right) \rightarrow P$ is injective (resp. weakly injective).
\end{definition}

It will become clear in a later section why we chose to use the term weakly independent atom configurations.
We can now reformulate Theorem \ref{thm:winjmap} in the setting of atom configurations.

\begin{proposition}
\label{p:Xnonemb}
If the poset $P$ contains a weakly independent atom configuration
$\{ x_{i,j}: i \in [d+1], j \in [3] \}$,
then $\Delta(\bar{P})$ is not embeddable in $\R^{2d}$.
\end{proposition}

\begin{proof}
By Definition~\ref{def:windconf}, the map $g$ corresponding to a weakly independent atom configuration is weakly injective. Also, $D_3^{*(d+1)}$ satisfies $\vartheta(D_3^{*(d+1)}) \neq 0$ and thus $\Delta(\bar{P})$ cannot be embedded in $\R^{2d}$ by Theorem \ref{thm:winjmap}.
\end{proof}

\begin{remark}
Clearly, every independent atom configuration is weakly independent by
definition. In most of the cases we encounter it would be sufficient for us 
to consider independent atom configurations. However, in case of projective
spaces over $\F_2$, weakly independent configurations will play a key role.

If $P$ contains an independent atom configuration, then $g \colon
\FF(D_3^{*(d+1)}) \rightarrow P$ is injective and $\tilde{g}$ is in fact an
isomorphism of $\sd D_3^{*(d+1)}$ and its image under $\tilde{g}$. Thus,
$\Delta(\bar{P})$ contains an isomorphic copy of the barycentric subdivision of
$D_3^{*(d+1)}$. This subcomplex is a straightforward certificate for the
non-embeddability of $\Delta(\bar{P})$ which does not use Proposition~\ref{p:Xnonemb}
and weakly independence, and therefore our presentation could be simplified if
we considered only the `independent' case. 
\end{remark}

We give a sufficient condition for atom configurations to be independent and weakly independent.
The following two lemmas combined with Proposition \ref{p:Xnonemb} are our main tools for showing non-embeddability in the rest of the paper.

\begin{lemma}
Let $\{ x_{i,j}: i \in [d+1], j \in [3] \} \subseteq \A(P)$ be an extendable atom configuration. If
\begin{equation}
\label{equ:indconf}
  x_{i,j} \not\leq \bigvee_{i \in [d+1]} x_{i,j_i}
\end{equation}
for any choice of $i \in [d+1]$ and $j,j_1,\ldots,j_{d+1} \in [3]$ with $j \not= j_i$, then the atom configuration is independent.
\end{lemma}

\begin{proof}
Assume that $\{ x_{i,j}: i \in [d+1], j \in [3] \} \subseteq \A(P)$ is an extendable atom configuration that satisfies condition \eqref{equ:indconf}.
We need to show that the corresponding map $g \colon \FF(D_3^{*(d+1)}) \rightarrow P$ is injective.
For that, let $\alpha, \beta \in \FF(D_3^{*(d+1)})$ be two simplices such that $\alpha \neq \beta$. 
Then we can find $v_{i,j} \in \alpha \setminus \beta$ (by possibly changing names of $\alpha$ and $\beta$).
We also find maximal faces $\sigma, \tau \in \FF(D_3^{*(d+1)})$ such that $\alpha \subseteq \sigma$, $\beta \subseteq \tau$ and $v_{i,j} \not\in \tau$.
Then condition \eqref{equ:indconf} implies that
$
 g(v_{i,j}) \not\leq g(\tau)
$.
However, $g(v_{i,j}) \leq g(\alpha)$ and $g(\beta) \leq g(\tau)$ which yields that $g(\alpha) \neq g(\beta)$. Thus, $g$ is injective and the atom configuration is independent.
\end{proof}

The criterion for weakly independent vector configurations is slightly more
technical to state. However, we invite the reader to check that it is an even
more immediate translation of the definition of a weakly extendable map to the
setting of atom configurations (and moreover even in an equivalent form).

\begin{lemma}
\label{l:char_wimap}
Let $\{ x_{i,j}: i \in [d+1], j \in [3] \} \subseteq \A(P)$ be an extendable
atom configuration. This atom configuration is weakly independent if and
only if for any choice of nonempty sets $I, I' \subseteq [d+1]$ and indices $j_i \in [3]$ for each $i \in I$ and $j'_i \in [3]$ for each $i \in I'$ satisfying $j_i \neq j'_i$ for every $i \in I \cap I'$ the following condition is satisfied:
\begin{equation}
\label{equ:windconf}
    \bigvee_{i \in I} x_{i,j_i} \neq 
    \bigvee_{i \in I'} x_{i,j'_i}.
\end{equation}
\qed
\end{lemma}

%We mention that we will mostly encounter independent atom configurations in our
%investigations. In that case, we can actually find a copy of the barycentric
%subdivision of $D_3^{*(d+1)}$ in the considered complex and our
%\red{presentation could} be simplified.
%
%However, in one important case, namely when $P$ is lattice of subspaces of a
%finite vector space over $\F_2$, we \red{need} weakly independent configurations.

%****************************************************************************
\section{Order complexes of finite projective spaces}
\label{s:projspace}
%****************************************************************************
%****************************************************************************

In this section, we apply our methods from the previous section and show non-embeddability of order complexes of the lattice of subspaces of a finite projective space. 

Let $\F_q$ be a finite field and let $\LL(\F_q^{d+2})$ be the geometric lattice of linear subspaces of $\F_q^{d+2}$ partially ordered by inclusion.
Set
$\Delta = \Delta(\bar{\LL(\F_q^{d+2})})$
to be
 the reduced order complex of $\LL(\F_q^{d+2})$. Thus, $\Delta$ is the simplicial complex whose maximal faces are given by complete flags of non-trivial subspaces of $\F_q^{d+2}$.

The most popular example is obtained from the Fano plane, that is the projective plane over $\F_2$.
%This is a one-dimensional building of type A.
We recall that the order complex of the Fano plane is shown in Figure~\ref{f:mapg}(b).

% where we have chosen a basis $e_1,e_2,e_3$ of $\F_2^3$ and labelled the vertices with the corresponding one- and two-dimensional subspaces.

%\begin{figure}
%\centering
%\includegraphics{fano.eps}
%\caption{The building associated to the Fano plane.}
%\label{fig:fano_plane}
%\end{figure}

\begin{theorem}
\label{t:typeA}
If $\Delta$ is the reduced order complex of the lattice of subspaces of the
finite projective space $\F^{d+2}_q$ with $d \geq 2$, then $\Delta =
\Delta(\bar{\LL(\F_q^{d+2})})$ does not embed into $\R^{2d}$.  
\end{theorem}

\begin{remark}
The $d=1$ case of Theorem \ref{t:typeA} corresponds to order complexes of finite projective planes.
Even though a complete classification of finite projective planes seems out of
reach, we still know that their order complexes are non-planar; see proof of
Theorem~\ref{t:buildings}~(i).
%Recall that the lattice of subspaces of a projective plane has rank $3$ and that it contains the points as atoms and the lines as coatoms. Thus, the corresponding order complex is a bipartite graph.
%Clearly, any such graph has girth at least six because two different points cannot be incident to two different lines and vice versa.
%
%It follows directly from Euler's formula that a planar graph of girth at least six has at least one vertex of degree at most two. So, if we start with a projective plane where each point is incident to at least three lines and each line contains at least three points, then the corresponding order complex will be a graph of minimal degree at least three and thus be non-planar.
\end{remark}

For proving Theorem \ref{t:typeA}, we construct (weakly) independent atom configurations in the lattice of subspaces $\LL(\F_q^{d+2})$.
Note that atoms correspond to points in the projective space or, equivalently, one-dimensional subspaces in the underlying finite vector space. The join of atoms is just the subspace that is spanned by the corresponding points.

Recall that $\langle \cdot \rangle$ denotes the subspace spanned by a set of points in a vector space. For simplicity, we will call a family of points
\[
   \big\{ x_{i,j}: i \in [d+1], j \in [3] \big\}
\]
in $\F_q^{d+2}$ a {\em (weakly) independent vector configuration} if the corresponding family of one-dimensional subspaces
\[
  \big\{ \langle x_{i,j} \rangle : i \in [d+1], j \in [3] \big\}
\]
is a (weakly) independent atom configuration in $\LL(\F_q^{d+2})$.
We make the following observations:
\begin{enumerate}
\item[(P1)]
  The join of any number of atoms exists in $\LL(\F_q^{d+2})$.
  Furthermore, every map $g : V(D_3^{*(d+1)}) \rightarrow \A(\LL(\F_q^{d+2}))$ is extendable.
\item[(P2)]
  Condition \eqref{equ:indconf} translates as follows to vector configurations: If
\begin{equation}
\label{equ:indconfvector}
  x_{i,j} \not\in \langle x_{i,j_i} : i \in [d+1] \rangle
\end{equation}
for any choice of $i \in [d+1]$ and $j,j_1,\ldots,j_{d+1} \in [3]$ with $j \not= j_i$, then $\{ x_{i,j}: i \in [d+1], j \in [3] \} \subseteq \A(P)$ is an independent vector configuration.
\item[(P3)]
  Similarly, Condition \eqref{equ:windconf} translates to: Let  be an extendable atom configuration. If
for any choice of nonempty sets $I, I' \subseteq [d+1]$ and indices $j_i \in [3]$ for each $i \in I$ and $j'_i \in [3]$ for each $i \in I'$ satisfying $j_i \neq j'_i$ for every $i \in I \cap I'$ the following condition is satisfied:
\begin{equation}
\label{equ:windconfvector}
    \langle x_{i,j_i} : i \in I \rangle \neq 
    \langle x_{i,j'_i} : i \in I' \rangle,
\end{equation}
then $\{ x_{i,j}: i \in [d+1], j \in [3] \}$ is a weakly independent vector configuration.
\end{enumerate}

We need to distinguish the case where $q = 2$ which will be treated in the latter part of this section. For now, assume that $q \geq 3$.

\begin{proposition}
\label{p:conftypeA}
Let $B = \{e_1,e_2,\ldots,e_{d+2}\}$ be any basis of $\F_q^{d+2}$ where $q \geq 3$.
Then the vector configuration
\[
    x_{i,j} = e_i + \lambda_j e_{d+2}, \quad i \in [d+1], j \in [3]
\]
is independent where $\lambda_j \in \F_q$ for each $j \in [3]$ are such that $\lambda_j \neq \lambda_{j'}$ if $j \neq j'$.
\end{proposition}

\begin{proof}
As stated in observation (P1), the given atom configuration is extendable.
As stated in (P2), we need to check that 
$$
    e_i + \lambda_j e_{d+2}
        \not\in \langle e_1 + \lambda_{j_1} e_{d+2}, \ldots,
        e_{d+1} + \lambda{j_{d+1}} e_{d+2} \rangle.
$$
for any choice of $i \in [d+1]$ and $j,j_1,\ldots,j_d \in [3]$ with $j \neq j_i$.
For contradiction, assume that $e_i + \lambda_j e_{d+2}$ is a nontrivial
linear combination of vectors from the right-hand side. Than no vector $e_k +
\lambda_{j_k} e_{d+2}$ with $k \neq i$ appears in this combination since it is the
only vector with nontrivial coefficient at $e_k$. Then $e_i + \lambda_j e_{d+2} \in
\langle e_i + \lambda_{j_i}  e_{d+2} \rangle$ which contradicts $j \neq j_i$.
\end{proof}

Let us now consider the special case where $q = 2$.
Thus, $\Delta = \Delta(\overline{\LL(\F_2^{d+2})})$ is the reduced order complex of the lattice of subspaces of $\F_2^{d+2}$.

\begin{proposition}
\label{p:conftypeA2}
Choose any basis $\{e_1,e_2,\ldots,e_{d+2}\}$ of $\F_2^{d+2}$. Then the vector configuration
\[
    x_{i,j} = u_i + w_j, \quad i \in [d+1],j \in [3]
\]
is weakly independent where $u_k = e_k$ for $k \in
[d]$, $u_{d+1} = e_1 + e_2$ as well as
$$
 w_1 = e_{d+1}, \quad w_2 = e_{d+2}, \quad w_3 = e_{d+1} + e_{d+2}.
$$
\end{proposition}

\begin{proof}
Again, the configuration is extendable by (P1).

For contradiction to (P3), assume that there are 
nonempty sets $I, I' \subseteq [d+1]$ and indices $j_i \in [3]$ for each $i \in I$ and $j'_i \in [3]$ for each $i \in I'$ satisfying $j_i \neq j'_i$ for every $i \in I \cap I'$ such that 
$$
    W :=
      \langle u_i + w_{j_i} : i \in I \rangle =
      \langle u_i + w_{j'_i} : i \in I' \rangle.
$$
We also set $J = I \cap \{ 1,2,d+1 \}$ and $J' = I' \cap \{ 1,2,d+1 \}$ as well as
$$
U = \langle x_{i,j_i} : i \in J \rangle \subseteq W
\quad \text{and} \quad
U' = \langle x_{i,j'_i} : i \in J' \rangle \subseteq W.
$$

We claim that $U = U'$. Let $x \in U \subseteq W$, then $x$ is expressible as a linear combination of the vectors $x_{i,j'_i}$ where $i \in I'$. However, no vector $x_{k,j'_k}$ can appear in this linear combination for $k \in I'
\setminus \{1,2,d+1\}$ since it would be the only vector with nonzero coefficient
at $e_k$ in the linear combination. Hence $x \in U'$ and thus $U \subseteq U'$. Similarly, we have $U' \subseteq U$ and thus $U = U'$ as claimed. Now, we need to distinguish several cases according to the dimension of $U$ and we will see that each case leads to a contradiction.

\emph{Case 1: $\dim U = 0$.} Then $J = J' = \emptyset$. If $k \in I$, then $x_{k,j_k} = u_k + w_{j_k} \in W$ and therefore $W$ contains a vector with a nonzero coefficient at $e_k$. This implies that $k \in I'$ as $x_{k,j'_k} = u_k + w_{j'_k}$ is the only possible vector $x_{i,j'_i}$ for $i \in I'$ with a nonzero coefficient at $e_k$. Similarly, if $k
\in I'$ then $k \in I$, and therefore $I = I'$. Consequently, for any
$k \in I$, we have
$$
(u_k + w_{j_k}) + (u_k + w_{j'_k}) = w_{j_k} + w_{j'_k} \in W.
$$
(Recall that we work over $\F_2$.)
However, the vector $w_{j_k} + w_{j'_k}$ cannot belong to $W$: It is a nonzero vector since $j_k \neq j'_k$ by our assumption. Also, it has zero coefficient at every $e_k$ with $k \in [d]$. This means that it cannot be expressed as a non-trivial linear combination of vectors $u_i + w_{j_i}$ for $i \in I$. A contradiction.

\emph{Case 2: $\dim U = 1$.} Then $|J| = |J'| = 1$ and $U$ and $U'$
are generated by linearly independent vectors which contradicts $U = U'$.

\emph{Case 3: $\dim U = 2$.} Then both $|J|$ and $|J'|$ contain at least two different elements. 
Therefore the set 
$$
  \Big\{ u_i + w_{j_i} : i \in J \Big\}
  \cup
  \left\{ u_i + w_{j'_i} : i \in J' \right\} \subseteq U
$$
contains at least four distinct non-zero vectors. However, the two-dimensional subspace $U$ of $\F_2^{d+2}$ can contain at most three non-zero vectors, a contradiction.

\emph{Case 4: $\dim U = 3$.} Then $J = J' = \{1,2, d+1\}$. In particular, each of
the following six different vectors belong to $U$:
$$
e_1 + w_{j_1}, \quad e_2 + w_{j_2}, \quad e_1 + e_2 + w_{j_{d+1}},
$$
$$
e_1 + w_{j'_1}, \quad e_2 + w_{j'_2}, \quad e_1 + e_2 + w_{j'_{d+1}}.
$$ 
Since there are only three possible vectors $w$, and since $j_1 \neq
j'_1; j_2 \neq j'_2$, there are $k_1 \in \{j_1,j'_i\}$ and $k_2 \in \{j_2,j'_2\}$ such that
$
w_{k_1} =  w_{k_2}.
$
Then the vector
$$
e_1 + w_{k_1} +  e_2 + w_{k_2} = e_1 + e_2
$$
belongs to $U$.
Similarly, we derive that $e_1$ and $e_2$ belong to $U$. Now, we have found
nine different vectors belonging to $U$ which contradicts that $\dim U = 3$. \qedhere
\end{proof}

\begin{proof}[Proof of Theorem \ref{t:typeA}]
By Proposition \ref{p:conftypeA} and Proposition \ref{p:conftypeA2},
$\F_q^{d+2}$ contains an independent or a weakly independent vector configuration for any $q \geq 2$ and $d \geq 2$, corresponding to an independent or weakly independent atom configuration in $\LL(\F_q^{d+2})$. The statement then follows from Proposition \ref{p:Xnonemb}.
\end{proof}

%****************************************************************************
\section{Order complexes of thick geometric lattices}
\label{s:geolatt}
%****************************************************************************
%****************************************************************************

This section generalizes the results of the last section for general geometric lattices. In order to ensure non-embeddability, we need to demand that the geometric lattices satisfy a thickness condition.

We say that a poset $P$ is {\em thick} if every open interval of length two contains at least three elements. Equivalently, the intervals in Figure~\ref{f:nonthick} may not appear in $P$. We note that the first interval in Figure~\ref{f:nonthick}, the length two chain, cannot appear in any geometric lattice.

\begin{figure}
\begin{center}
\includegraphics{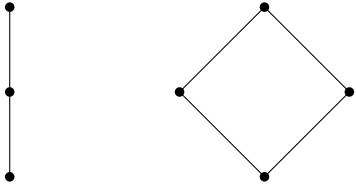}
\end{center}
\caption{Length two intervals which cannot occur in a thick lattice.}
\label{f:nonthick}
\end{figure}

\begin{theorem}
\label{t:geometric}
If $L$ is a finite thick geometric lattice of rank $d+2$
then the ($d$-dimensional) reduced order complex $\Delta(\bar{L})$ does not embed into $\R^{2d}$.
\end{theorem}

\begin{remark}
There is a property for lattices called {\em relatively complemented} which we do not define here.
In \cite[Theorem 2]{Bjoerner1981}, it is shown that for a lattice of finite length, being relatively complemented is equivalent to the absence of $3$-element intervals. Thus, every thick lattice of finite length is automatically relatively complemented.
Furthermore, a semimodular lattice is relatively complemented if and only if it is atomistic \cite[Theorem 6]{Birkhoff1995}. This implies that in fact, any thick semimodular lattice is atomistic and thus geometric. So in Theorem \ref{t:geometric} we could demand that $L$ is semimodular instead. This would not make the result any more general though.
\end{remark}

Modular geometric lattices have a very special structure as is revealed by the following statement:

\begin{lemma}[{\cite[Theorem 10]{Birkhoff1995}}]
\label{lem:product}
Any modular geometric lattice is a product of a Boolean algebra with projective geometries.
\end{lemma}

In a graded lattice, an element of rank $2$ is a {\em line} and an
element of corank $1$ is a {\em hyperplane}. It will prove useful for us that modular lattices can be characterized by a relation between lines and hyperplanes.

\begin{lemma}[{\cite[Lemma 23.8]{vanLintWilson2001}}]
\label{lem:nonmod}
A finite geometric lattice $L$ is modular if and only if every line and every hyperplane meet non-trivially.
\end{lemma}

We show Theorem~\ref{t:geometric} by induction on the rank of $L$. We need to
distinguish two cases concerning the modularity of $L$. For both cases, we
assume that Theorem~\ref{t:geometric} is already shown for lower-rank lattices.

\begin{proposition}
\label{p:modular}
%\label{thm:main}
If $L$ is a thick modular geometric lattice of rank $d+2$
then the reduced order complex $\Delta(\bar{L})$ does not embed into $\R^{2d}$.
\end{proposition}

\begin{proof}
Following Lemma \ref{lem:product}, we see that $L$ is the product of a Boolean lattice and projective geometries. However, the Boolean lattice is not thick, so this must be a trivial factor. Thus, $L$ is of the form
\[
    L = L_1 \times \cdots \times L_m
\]
where each $L_k$ is the subspace lattice of a projective geometry. By Proposition \ref{p:conftypeA} and Proposition \ref{p:conftypeA2}, each $L_k$ contains a weakly injective atom configuration
\[
  \left\{ a^k_{i,j} : i \in [\rk(L_k)], j \in [3] \right\}.
\]
We claim that the union of those configurations forms a weakly independent atom configuration in $L$. Consider the atom configuration
\[
  \left\{ (a^1_{i,j},\hat{0},\ldots,\hat{0}) \right\}
  \cup
  \left\{ (\hat{0},a^2_{i,j},\hat{0},\ldots,\hat{0}) \right\}
  \cup
  \ldots
  \cup
  \left\{ (\hat{0},\ldots,\hat{0},a^m_{i,j}) \right\}
\]
where in each set, the elements $a^k_{i,j}$ are for $i \in [\rk(L_k)]$ and $j \in [3]$.
Using Lemma~\ref{l:char_wimap}, it is not hard to check that this actually is a weakly independent atom
configuration. By Proposition \ref{p:Xnonemb}, $\Delta(\bar{L})$ cannot be embedded into $\R^{2d}$.
\end{proof}

Next, we consider the non-modular case. Note that a non-modular geometric lattice must have rank at least $3$.

\begin{proposition}
\label{p:nonmodular}
If $L$ is a thick non-modular geometric lattice of rank $d+2$
then the reduced order complex $\Delta(\bar{L})$ does not embed into $\R^{2d}$.
\end{proposition}

\begin{proof}
By Lemma \ref{lem:nonmod}, we find a hyperplane $h$ and a line $\ell$ in $L$ such that $h \wedge \ell = \hat{0}$.

The interval $L' = [\hat{0},h]$ is a finite thick geometric lattice of smaller rank. By induction, we find a weakly independent atom configuration
\[
    \left\{ a_{i,j} : i \in [d], j \in [3] \right\}
\]
in $L'$.

Because $L$ is thick, the line $\ell$ contains at least three points and we can choose three different atoms $a_{d+1,1},a_{d+1,2},a_{d+1,3}$ covered by $\ell$.
It remains to show that
\[
    \{ a_{i,j} \}_{i \in [d+1]; j \in [3]}
\]
is a weakly independent atom configuration in $L$. For that, choose $I, I' \subseteq [d+1]$ and $j_i \in [3]$ for each $i \in I$ as well as $j'_{i'} \in [3]$ for each $i' \in I'$ such that $j_i \neq j'_i$ if $i \in I \cap I'$. Assume that
\[
    z := \bigvee_{i \in I} a_{i,j_i} = \bigvee_{i' \in I'} a_{i',j'_{i'}}
\]
and set
\[
    x := \bigvee_{i \in I \setminus \{d+1\}} a_{i,j_i} \quad \text{and} \quad
    x' := \bigvee_{i \in I' \setminus \{d+1\}} a_{i,j'_i}
\]
If $d+1 \not\in I$, then $x = x' \vee a_{d+1,j'_{d+1}}$. However, $x \leq h$ and then $a_{d+1,j'_{d+1}} \leq h$ in contradiction to $h \wedge \ell = \hat{0}$. Thus, $d+1 \in I$. The same argument proves that $d+1 \in I'$ and we get that $d+1 \in I \cap I'$.

We see that $x \leq z$ and $x' \leq z$, so $x \leq x' \vee x \leq z$. However, $z = x \vee a_{r,j_r}$ and
thus $z$ must cover $x$
(by semimodularity, taking the join with an atom can increase the rank by at most one).
However, $z \neq x \vee x'$ because otherwise $a_{d+1,j_{d+1}} \leq z = x \vee x' \leq h$ in contradiction to $h \cap \ell = \hat{0}$. Thus, $x \vee x' = x$ and by the same argument $x \vee x' = x'$. This is a contradiction because the configuration in $L'$ is weakly independent.
\end{proof}

Theorem \ref{t:geometric} follows from Proposition \ref{p:modular} and Proposition \ref{p:nonmodular}.

%****************************************************************************
\section{Order complexes of some finite buildings}
\label{s:build}
%****************************************************************************
%****************************************************************************

In this section, we show that several classes of $d$-dimensional finite buildings cannot be embedded into $\R^{2d}$. This includes finite buildings that are one-dimensional or of type A as well as two classes of finite buildings of type B.

\begin{theorem}
\label{t:buildings}
A $d$-dimensional thick
building $\Delta$ does not embed into $\R^{2d}$ if any of the following
conditions is satisfied
\begin{enumerate}
\item $d = 1$,
\item $\Delta$ is of type $A$,
\item $\Delta$ is of type $B$ coming from an
alternating bilinear form on $\F_q^{2d+2}$, or
\item $\Delta$ is of type $B$ coming from an
Hermitian form on $\F^{2d + 2}_{q^2}$ or $\F^{2d + 3}_{q^2}$.
\end{enumerate}
\end{theorem}

We begin with a short introduction to the necessary background about finite buildings. This includes a remark why we have already shown part (i) and (ii) of Theorem \ref{t:buildings}. Part (iii) and (iv) will be proved in Section \ref{sec:typeBalternating} and Section \ref{sec:typeBhermitian}, respectively.

%****************************************************************************
\subsection{Finite buildings}
\label{s:buildings}
%****************************************************************************

This subsection gives a very short introduction to finite buildings. For more details, the interested reader should have a look at \cite{AbramenkoBrown2008}, \cite{Ronan2009} or \cite{Tits1974}. 

A \emph{finite Coxeter complex} is a simplicial subdivision of a sphere induced by a hyperplane arrangement corresponding to a finite reflection group, see \cite[Chapter 3]{AbramenkoBrown2008} or \cite{Humphreys1990}.
A \emph{finite building} is a finite simplicial complex $\Delta$ which is glued
together as a union of Coxeter complexes $\Gamma$, called \emph{apartments}, 
following certain axioms.
These axioms are simple but imply that buildings are structures of high complexity and with large symmetry groups.

A buildings is \emph{thick} if each face of codimension one is contained in at least three maximal faces.
A building is \emph{irreducible} if it is not isomorphic to the simplicial join of two smaller buildings.

One-dimensional finite buildings are known as \emph{generalized $m$-gons}. They
are finite bipartite graphs of diameter $m$ and girth $2m$ for some $m \geq 3$
\cite[Proposition 4.44]{AbramenkoBrown2008}. This means that every
shortest path between any two vertices has length at most $m$ and that any
cycle has length at least $2m$. 
%No simple characterization of one-dimensional buildings is known.  However,
%It is not hard to see that generalized $m$-gons
%are non-planar graphs, see Remark \ref{rem:genmgon} for more details.

Irreducible finite buildings of dimension at least two are classified according to the types of the Weyl groups of the underlying Coxeter complexes. Our results cover finite buildings of type A and some finite buildings of type B, which can all be described very explicitly as order complexes of certain posets.

% \subsubsection*{Finite buildings of type A}

Every finite thick building $\Delta$ of type A and dimension $d
\geq 2$ is isomorphic to the order complex of the poset of proper subspaces of a $(d+2)$-dimensional vector space over a finite field $\F_q$ (see \cite{Tits1974}). We discuss embeddability of those complexes in Section \ref{s:projspace}.

%\subsubsection*{Finite buildings of type B}
%\label{s:typeBforms}

Every finite thick building $\Delta$ of {\em type B} and dimension $d \geq 2$ is obtained from a {\emph{classical polar space} of rank $d+1$. (For an axiomatic description of polar spaces in general, see \cite{Tits1974}.)
% We won't use the axioms for our results and give a different description following \cite[Chapter 9.3]{AbramenkoBrown2008}.
All such $\Delta$ are order complexes of posets of totally isotropic subspaces
associated to forms of Witt index $d+1$ on vector spaces over finite fields, partially ordered by inclusion.
The forms that yield thick finite and irreducible buildings of type $B$ and dimension $d \geq
2$ are
alternating bilinear forms on $\F_q^{2d+2}$,
Hermitian forms on $\F_{q^2}^m$ for $m = 2d+2$ or $m = 2d+3$,
symmetric bilinear forms on $\F_q^{2d+3}$
and
non-degenerate quadratic forms on $\F_q^{2d+4}$
(see \cite[Chapter 9.3]{AbramenkoBrown2008} for details).
In all cases, the corresponding totally isotropic subspaces of the respective vector space are those subspaces on which the respective form vanishes constantly.

% We discuss the embeddability of finite buildings of type B coming from alternating or Hermitian forms in Section \ref{???}.

\begin{proof}[Proof of Theorem~\ref{t:buildings} (i) and (ii)]
%\begin{remark}
%\label{rem:genmgon}
Recall that all one-dimensional buildings, called generalized $m$-gons, are
bipartite graphs and thus order complexes of posets of rank $2$. It is not hard
to show that in fact all thick generalized $m$-gons are non-planar graphs:
If $m \geq 3$ then a thick generalized $m$-gon has girth at least six and minimal degree at least three. But by Euler's formula, every planar graph of girth at least six has at least one vertex of degree at most two.
If $m=2$ it is not hard to see that generalized $2$-gons are exactly complete
bipartite graphs $K_{p,q}$ with $p, q \geq 2$ (this also follows from the fact
that generalized $2$-gons, as buildings, are reducible since their Coxeter
diagram is not connected). Since we consider thick buildings, we have $p,q \geq
3$ and it is well known that any graph containing $K_{3,3}$ is non-planar.

Part (ii) of Theorem~\ref{t:buildings} was already implicitly proved since
every finite building of type A is obtained as the order complex of the
lattice of subspaces of a finite vector space and the embeddability of those
has been treated in Section \ref{s:projspace}.
%\begin{remark}
%Parts (i) and (ii) of Theorem \ref{t:buildings} were already implicitly proved:
%
%First, we already stated in Remark \ref{rem:genmgon} why finite one-dimensional thick buildings are non-planar graphs.
%
%Second, every finite buildings of type A is obtained as the order complex of the lattice of subspaces of a finite vector space and the embeddability of those has been treated in Section \ref{s:projspace}.
%\end{remark}
\end{proof}

%\red{
%\martin{I have moved Remark 4.5 (in previous notation) to section 7 letting it
%in a proof environment.}
%}

\begin{remark}
Theorem \ref{t:buildings} implies non-embeddability also for many non-irreducible buildings:
Suppose that $\Delta$ is a building which is not irreducible.
We may assume
that $\Delta = \Delta_1 * \Delta_2$ where $\Delta_1$ and $\Delta_2$ are already
irreducible (otherwise we proceed by induction).
If our results apply to $\Delta_1$ and $\Delta_2$ because they are of right type, then can exhibit a map
$f_i \colon |K_i| \rightarrow |\Delta_i|$ which satisfies condition \eqref{eq:cond} where
$K_i = D_3^{*(\dim \Delta_i + 1)}$.
We use Lemma~\ref{l:joins} and find
that there is a map $f\colon |K| \rightarrow |\Delta|$, again satisfying Condition \eqref{eq:cond},
where $K = K_1 * K_2 = D_3^{*(\dim \Delta_1 + \dim \Delta_2 + 2)}$. Using
Proposition~\ref{p:map} we see that $\Delta$ does not embed into $\R^{2\dim \Delta}$,
note that $\dim \Delta = \dim \Delta_1 + \dim \Delta_2 + 1$.
\end{remark}

It remains for us to apply our methods to two classes of finite buildings of type B. Recall that those are obtained as order complexes of the poset of totally isotropic subspaces of some finite polar space. We treat the cases where the form defining those subspaces is either an alternating bilinear form or a Hermitian form.

%****************************************************************************
\subsection{Buildings of type B coming from alternating bilinear forms}
\label{sec:typeBalternating}
%****************************************************************************

If $\Delta$ is a finite thick $d$-dimensional building of type B coming from an alternating bilinear form, then $\Delta$ is the reduced order complex of the poset of totally isotropic subspaces corresponding to
that form $\bil{\cdot, \cdot}$ on a finite vector space $\F_q^{2d+2}$.

We can find basis vectors
$
e_1,\ldots,e_{d+1},f_1,\ldots,f_{d+1}
$
such that $\bil{\cdot, \cdot}$ is defined by
\begin{equation}
\label{eq:symplform}
  \bil{e_i,f_i} = - \bil{f_i,e_i} = 1
\end{equation}
for $i=1,\ldots,d+1$ and such that all other ``inner products'' between basis vectors are zero
\cite[Remark 6.99]{AbramenkoBrown2008}.

Let $P$ be the poset of all totally isotropic subspaces of $\F_q^{2d+2}$, that means those subspaces on which the bilinear form $\bil{\cdot,\cdot}$ vanishes constantly, ordered by inclusion.
The maximal dimension of such a subspace in this setting is known to be $d+1$, so $P$ has rank $d+1$.
Also, there is no unique maximal subspace.

% \begin{theorem}
% \label{thm:typeBalternating}
% If $\Delta$ is a finite thick $d$-dimensional building of type B with $d \geq 2$ which is the order complex of the poset $P$ of totally isotropic subspaces corresponding to
% an alternating bilinear form on $\F_q^{2d+2}$,
% then $\Delta$ cannot be embedded in $\R^{2d}$.
% \end{theorem}

As before, we construct an independent vector configuration in $\F_q^{2d+2}$ - corresponding to an independent atom configuration in $P$.

\begin{proposition}
\label{p:sympl}
Let $e_1,e_2,\ldots,e_{d+1},f_1,\ldots,f_{d+1}$ be a basis of $\F_q^{2d+2}$ and let $\bil{\cdot,\cdot}$ be the bilinear form on $\F_q^{2d+2}$ defined as in Equation \eqref{eq:symplform}. Then the vector configuration
\[
    x_{i,1} = e_i, \quad x_{i,2} = e_i + f_i, \quad x_{i,3} = f_i \quad i \in [d+1]
\]
is independent (with respect to $P$).
\end{proposition}

\begin{proof}
First, we need to show that the vector configuration is extendable or equivalently, that
$$
   \langle x_{1,j_1},\ldots,x_{d+1,j_{d+1}} \rangle
$$
is a totally isotropic subspace for any choice of $j_1,\ldots,j_{d+1} \in [3]$. But this is true because
$$
    \bil{ x_{i,j}, x_{i',j'} } = 0
$$ 
for any choice of $i,i' \in [d+1]$ and $j,j' \in [3]$ with $i \neq i'$ and we also check that
$
    \bil{ x_{i,j}, x_{i,j} } = 0
$
for any $i \in [d+1]$ and $j \in [3]$.

Let now $i \in [d+1]$ and $j,j_1,\ldots,j_{d+1} \in [3]$ with $j \neq j_i$ and assume that
$$
    x_{i,j} \in \langle x_{1,j_1}, \ldots, x_{d+1,j_{d+1}} \rangle.
$$
Then in particular 
$\bil{ x_{i,j}, x_{i,j_i} } = 0$ 
needs to be satisfied. However, we see that this can only be true if $j = j_i$, a contradiction.
\end{proof}

Proposition \ref{p:sympl} implies part (iii) of Theorem \ref{t:buildings}.

%****************************************************************************
\subsection{Buildings of type B coming from Hermitian forms}
\label{sec:typeBhermitian}
%****************************************************************************

If $\Delta$ is a finite thick $d$-dimensional building of type B coming from a Hermitian form, then it is the reduced order complex of the poset of totally isotropic subspaces corresponding to that form $\bil{\cdot, \cdot}$ on $\F_{q^2}^{m}$ for $m = 2d+2$ or $m = 2d+3$.

Being Hermitian means that $\bil{\cdot,\cdot}$ is linear in the first argument and satisfies
\[
    \bil{ w, v }
        = \overline{\bil{ v, w }} 
\]
for all $v,w \in \F_{q^2}^m$. So it is conjugate-linear in the second argument.

Then there are basis vectors $e_1,\ldots,e_n,f_1,\ldots,f_n$ for $m$ even or
basis vectors $e_1,\ldots,e_n,e_{n+1}$, $f_1,\ldots,f_n$ for $m$ odd such that $\bil{\cdot,\cdot}$ is given by
\begin{equation}
 \label{eq:hermform}
 \bil{e_i,f_i} = \bil{f_i,e_i} = 1
\end{equation}
and all other ``inner products'' of basis vectors are zero \cite[Remark 6.104]{AbramenkoBrown2008}.
(If $m$ is odd, we also have $\bil{e_{n+1},e_{n+1}} = 1$.)

Again, let $P$ be the poset
of totally isotropic subspaces of $\F_{q^2}^m$ with respect to $\bil{\cdot, \cdot}$,
that is those on which the Hermitian form vanishes constantly. Also in this case, the maximal dimension of a totally isotropic subspace is known to be $d+1$
and there is no unique maximal such.

% \begin{theorem}
% \label{thm:typeBhermitian}
% If $\Delta$ is a finite thick $d$-dimensional building of type B with $d \geq 2$
% which is the order complex of the poset of totally isotropic subspaces
% corresponding to a Hermitian form on $\F_{q^2}^{m}$ for $m = 2d+2$ or $m = 2d+3$,
% then $\Delta$ cannot be embedded in $\R^{2d}$.
% \end{theorem}

We need to go through some details about $\F_{q^2}$:
The finite field $\F_{q^2}$ has a conjugation, that means it is a
degree two extension of another finite field. We think of $\F_{q^2}$ as
$\F_q[x]/(x^2+k)$ where $k \in \F_q$ is any non-square element. Then a
conjugation of $\F_{q^2}$ is given by the map that sends elements of $\F_q$ to
themselves and $x$ to $-x$. We find that the conjugation sends the element $\alpha + \beta x$ to $\alpha - \beta x$ for any $\alpha, \beta \in \F_q$. In particular,  we see that
it sends $\beta x$ to $- \beta x$ for any $\beta \in \F_q$. We write $\overline{\lambda}$ for the conjugation of $\lambda \in F_{q^2}$.

\begin{lemma}
\label{l:conjmin}
In a finite field $\F_{q^2}$ with a conjugation, there exists some $\lambda \neq 0$ such that $\overline{\lambda} = -\lambda$.
\end{lemma}

\begin{proof}
Choose $\lambda \in \F_{q^2}$ which corresponds to $x$ in $\F_q[x]/(x^2+k)$.
\end{proof}

Also in this case, we can construct an independent vector configuration.

\begin{proposition}
\label{p:unitary}
Let $e_1,\ldots,e_{\lfloor m/2-1 \rfloor},f_1,\ldots,f_{d+1}$ be a basis of
$\F_{q^2}^m$ with $m = 2d+2$ or $m=2d+3$ such that the non-degenerate Hermitian form $\bil{\cdot,\cdot}$ is given by Equation \eqref{eq:hermform}. Let $\lambda \in F_{q^2}$ be such that $\overline{\lambda} = - \lambda$. Then the vector configuration
\[
    x_{i,1} = e_i, \quad x_{i,2} = e_i + \lambda f_i, \quad x_{i,3} = f_i, \quad i \in [d+1]
\]
is independent (with respect to the poset $P$).
\end{proposition}

\begin{proof}
As for the symplectic case, we need to show that the vector configuration is extendable, meaning that the Hermitian form vanishes on each subspace of the form
$$
    \langle x_{1,j_1},\ldots,x_{d+1,j_{d+1}} \rangle
$$
for $j_1,\ldots,j_{d+1} \in [3]$. This is true because 
$\bil{ x_{i,j}, x_{i',j'}} = 0$
for $i \neq i'$ and also 
$\bil{ x_{i,j}, x_{i,j} } = 0$
for each $i \in [d+1]$ and $j \in [3]$.

Now, we claim that the vector configuration is independent. However, this can be shown using the same arguments as in the proof of Proposition \ref{p:sympl}.
\end{proof}

Proposition \ref{p:unitary} implies part (iv) of Theorem \ref{t:buildings}.

%****************************************************************************
\section{Outlook}
\label{s:outlook}
%****************************************************************************
%****************************************************************************

We think that the complex structure of finite thick buildings justifies the following conjecture. Our results confirm the conjecture for several large classes of buildings.

\begin{conjecture}
\label{c:main}
No $d$-dimensional finite thick building embeds into $\R^{2d}$.
\end{conjecture}

It would be very desirable to prove Conjecture~\ref{c:main} using an argument which works for all finite thick buildings. As for example buildings of type D are not order complexes of any posets, our method won't work for the general case. However, we know that finite thick buildings of type D are closely related to the also still unsolved case of finite thick buildings of type B which come from a quadratic form.
%It might suffice to prove the conjecture in type B and this might be possible to do with our methods.

Another possible approach to prove Conjecture~\ref{c:main} is to use root group techniques as suggested by Bernd M\"{u}hlherr. It seems reasonable to believe that all root groups corresponding to finite thick buildings contain a subrootgroup corresponding to the product of rank one type $A$ root groups in such a way that the buildings all contain the $(d+1)$-fold join of disjoint points or some subdivision of it as a subcomplex.

%It would also be interesting to know whether the results of Section \ref{s:projspace} could be generalized to an even larger class of posets.

%****************************************************************************
\section*{Acknowledgements}
%****************************************************************************
%****************************************************************************

We would like to thank Anders Bj\"{o}rner, Marek Kr\v{c}\'al, Bernd M\"{u}hlherr, Joseph A. Thas and Uli Wagner for helpful discussions
about this project.

\bibliographystyle{alpha}
\bibliography{References}

\begin{thebibliography}{MTW11}

\bibitem[AB08]{AbramenkoBrown2008}
Peter Abramenko and Kenneth~S. Brown.
\newblock {\em Buildings}, volume 248 of {\em Graduate Texts in Mathematics}.
\newblock Springer, New York, 2008.
\newblock Theory and applications.

\bibitem[Bir79]{Birkhoff1995}
Garrett Birkhoff.
\newblock {\em Lattice theory}, volume~25 of {\em American Mathematical Society
  Colloquium Publications}.
\newblock American Mathematical Society, Providence, R.I., third edition, 1979.

\bibitem[Bj{\"o}81]{Bjoerner1981}
Anders Bj{\"o}rner.
\newblock On complements in lattices of finite length.
\newblock {\em Discrete Math.}, 36(3):325--326, 1981.

\bibitem[Flo34]{flores32}
A.~Flores.
\newblock {\"U}ber n-dimensionale {K}omplexe die im ${R}_{2n+1}$ absolut
  selbstverschlungen sind.
\newblock {\em Ergeb. Math. Kolloq.}, 4:6--7, 1932/1934.

\bibitem[Hum90]{Humphreys1990}
J.~E. Humphreys.
\newblock {\em Reflection groups and {C}oxeter groups}, volume~29 of {\em
  Cambridge Studies in Advanced Mathematics}.
\newblock Cambridge University Press, Cambridge, 1990.

\bibitem[Mat03]{Matousek2003}
Ji{\v{r}}{\'i} Matou\v{s}ek.
\newblock {\em Using the {B}orsuk-{U}lam theorem}.
\newblock Universitext. Springer-Verlag, Berlin, 2003.

\bibitem[Mel09]{melikhov09}
S.~A. Melikhov.
\newblock The van {K}ampen obstruction and its relatives.
\newblock {\em Tr. Mat. Inst. Steklova}, 266(Geometriya, Topologiya i
  Matematicheskaya Fizika. II):149--183, 2009.

\bibitem[MTW11]{matousek-tancer-wagner11}
J.~Matou\v{s}ek, M.~Tancer, and U.~Wagner.
\newblock Hardness of embedding simplicial complexes in {$\mathbb R^d$}.
\newblock {\em J. Eur. Math. Soc. (JEMS)}, 13(2):259--295, 2011.

\bibitem[Ron09]{Ronan2009}
Mark Ronan.
\newblock {\em Lectures on buildings}.
\newblock University of Chicago Press, Chicago, IL, 2009.
\newblock Updated and revised.

\bibitem[Sha57]{shapiro57}
A.~Shapiro.
\newblock Obstructions to the imbedding of a complex in a euclidean space. {I:
  The} first obstruction.
\newblock {\em Ann. of Math., II. Ser.}, 66:256--269, 1957.

\bibitem[Tit74]{Tits1974}
Jacques Tits.
\newblock {\em Buildings of spherical type and finite {BN}-pairs}.
\newblock Lecture Notes in Mathematics, Vol. 386. Springer-Verlag, Berlin,
  1974.

\bibitem[vK32]{vankampen32}
R.~E. van Kampen.
\newblock Komplexe in euklidischen {R}{\"{a}}umen.
\newblock {\em Abh. Math. Sem. Univ. Hamburg}, 9:72--78, 1932.
\newblock Berichtigung dazu, \emph{ibid}. (1932) 152--153.

\bibitem[vLW01]{vanLintWilson2001}
J.~H. van Lint and R.~M. Wilson.
\newblock {\em A course in combinatorics}.
\newblock Cambridge University Press, Cambridge, second edition, 2001.

\bibitem[Wag11]{Wagner2011}
Uli Wagner.
\newblock Minors in random and expanding hypergraphs.
\newblock In {\em Proceedings of the 27th annual ACM symposium on Computational
  geometry}, SoCG '11, pages 351--360, New York, NY, USA, 2011. ACM.

\bibitem[Whi86]{White1986}
Neil White.
\newblock {\em {T}heory of {M}atroids}.
\newblock Cambridge University Press, 1986.

\bibitem[Wu65]{wu65}
W.-T. Wu.
\newblock A theory of imbedding, immersion, and isotopy of polytopes in a
  {E}uclidean space.
\newblock Science Press, Peking, 1965.

\end{thebibliography}
%\bibliography{martinsreferences}
%DO NOT ERASE THE LINE BELOW PLEASE.
%\bibliography{/home/martin/clanky/bib/general}

\end{document}